\documentclass{amsart}
\usepackage{amsfonts}
\usepackage{}
\usepackage{amsfonts,amssymb,amsmath,amsthm}
\usepackage{url}
\usepackage{enumerate}
\usepackage[dvips]{epsfig}
\usepackage{mathrsfs}
\usepackage{latexsym}
\usepackage{bbm}

\newtheorem{lemma}{Lemma}[section]
\newtheorem{theorem}[lemma]{\bf  Theorem}
\newtheorem{corollary}[lemma]{ \bf Corollary}
\newtheorem{proposition}[lemma]{  \bf Proposition}
\newtheorem{remark}[]{  \bf Remark}
\newtheorem{example}[]{  \bf Example}
\newtheorem{definition}[lemma]{  \bf Definition}

\DeclareMathOperator{\diam}{diam} 
\DeclareMathOperator{\vol}{Vol}\DeclareMathOperator{\id}{id}
\DeclareMathOperator{\Exp}{Exp^c}

\DeclareMathOperator{\Sp}{Span_{\mathbb{R}}}
\DeclareMathOperator{\E}{E}

\begin{document}

\title{A comparison theorem for Finsler submanifolds and its applications}

\author{Wei Zhao}
\address{
 Center of Mathematical Sciences\\
Zhejiang University\\
Hangzhou, China}
\email{zhaowei008@yahoo.cn}
\thanks{This work was supported partially
by National Natural Science Foundation of China (Grant No.
11171297).}

\begin{abstract}
In this paper, we consider the conormal bundle over a submanifold in a Finsler manifold and establish a volume comparison theorem. As an application, we derive a lower estimate for length of closed geodesics in a Finsler manifold. In the reversible case, a lower bound of injective radius is also obtained.
\end{abstract}
\maketitle

\section{Introduction}
In Riemannian geometry it is an important subject to investigate the geometry of submanifolds. And there are many important local and global results, which in turn lead to a better understanding on Riemann manifolds\cite{Che,Fr,HK,L,PT,Wi}. For example,
the Heintze-Karcher comparison theorem\cite{HK}
plays a very important role in the global differential geometry
of Riemann manifolds, which says that the upper bound for the volume of a closed Riemannian manifold can be estimated in terms of the the volume and the mean curvature of an arbitrary closed submanifold, the diameter and a lower bound for the section curvature. One of it applications is the lower estimate on the length of sample closed geodesics in a closed Riemannian manifold\cite{IC,HK}. More precisely, if $(M,g)$ is a closed Riemannian $m$-manifold with the section curvature $\mathbf{K}\geq \delta$, $\diam(M)\leq d$ and $\vol(M)\geq V$, then for any sample closed geodesic $\gamma$ in $M$, its length satisfies
\[
L_g(\gamma)\geq \frac{(m-1)V}{c_{m-2}\mathfrak{s}_\delta^{m-1}\left(\min\left\{d,\frac{\pi}{2\sqrt{\delta}}\right\}\right)},
\]
where $c_m:=\vol(\mathbb{S}^m)$, $\pi/\sqrt{\delta}:=+\infty$ if $\delta\leq 0$, and $\mathfrak{s}_\delta(t)$ is the unique solution to
$y''+\delta y=0$ with $y(0)=0$ and $y'(0)=1$. Combining this with Kingenberg's theorem\cite{Kl}, one can obtain the injectivity radius estimate of Cheeger\cite{CH} without using Toponogov's comparison theorem. See \cite{Bu,GR,M,MJ,Sc}, etc., for more details on the Heintze-Karcher comparison theorem.

Finsler geometry, a natural generalization of Riemannian geometry, was initiated
by Finsler\cite{F} from considerations of regular problems in the calculus
of variations. Recently, the geometry of Finsler submanifolds has been developed tremendously, especially in the aspects of Finsler minimal submanifolds and Minkowski submanifolds\cite{HS,Sh1,Sh5,ST}. However, the geometry of Finsler submanifolds is much different from the one of Riemannian submanifolds. For instance, there exit totally geodesic submanifolds which are not minimal for the
Busemann-Hausdorff measure but are minimal for the Holmes-Thompson measure\cite{AB}.

Now we consider the normal bundle of a Finsler submanifold.
Let $(M,F)$ be a forward complete Finsler $m$-manifold and let $N$ be a connected $k$-dimensional submanifold of $M$, $0\leq k\leq m$. According to \cite{Ru,Sh2}, the "normal bundle" $\mathcal {V}N$ of $N$ in $M$ is defined as $\mathcal {V}N:=\cup_{x\in N}\mathcal {V}_xN$, where $\mathcal {V}_xN:=\{0\}\cup\{n\in T_xM: n\neq 0, \,g_{n}(n,X)=0, \, \forall X\in T_xN\}$.
It is a generally recognized principle that the normal bundle $\mathcal {V}N$ of $N$ in $M$ carries much geometric information. However,
in general case, $\mathcal {V}N$ is not a vector bundle but a cone bundle\cite{Ru,Sh2}. Apparently, it is rather hard to handle due to nonlinearity of $\mathcal {V}N$. Recall the Legendre transformation $\mathcal {L}:TM\rightarrow T^*M$ is a homeomorphism\cite{BCS,Sh1}. It should be remarked that $\mathcal {L}$ is a diffeomorphism
(or isomorphism) if and only if $F$ is Riemannian. The {\it conormal bundle} $\mathcal {V}^*N$ of $N$ in $M$ is defined as the homeomorphic image of $\mathcal {V}N$ under $\mathcal {L}^{-1}$. Clearly, $\mathcal {V}^*N$ coincides with the original definition in the Riemannian case.
The study of the (co-)normal bundle of a Finsler submanifold is still at its infant stage. Several people have made some fundamental contributions to this subject
from various points of view\cite{Be,Da,Ma,Ru,Sh1,Sh5}, etc.

The purpose of this paper is to investigate the conormal bundle of a Finsler submanifold and establish the Heintze-Karcher comparison theorem in Finsler geometry for both the Busemann-Hausdorff measure and the Holmes-Thompson measure.

Given $\xi\in \mathcal {V}^*N\backslash0$, let $H_\xi$ denote the {\it co-mean curvature} along $\xi$, which is given explicitly in Sec.3. First, we have the following theorem.
\begin{theorem}
Let $(M,F)$ be a closed Finsler $m$-manifold with uniform constant $\Lambda_F$ and diameter $d$ and let $N$ be a connected $k$-dimensional submanifold of $M$, $0\leq k\leq m$. Suppose the flag curvature $\mathbf{K}\geq \delta$.

(1) If $k=0$, i.e., $N=\{x\}$, then
\[
\mu(M)\leq \int_{S_xM}e^{-\tau(\dot{\gamma}_y(t))}d\nu_x(y) \int^{d}_0\mathfrak{s}^{m-1}_\delta(t)dt,
\]
where $\mu$ is any volume form on $M$ and $d\nu_x$ is the Riemannian volume on $S_xM$ induced by $g_x$.

(2) If $k\geq 1$, then
\[
\mu(M)\leq c_{m-k-1}\cdot\Lambda_F^{(3m+k)/2}\cdot\bar{\mu}(N)\cdot\int^{\min\left\{d,\,\zeta(\xi_0)\right\}}_0
\left(\mathfrak{s}'_\delta-\frac{H_{\xi_0}}{k}\mathfrak{s}_\delta\right)^k(t)\cdot\mathfrak{s}^{m-k-1}_\delta(t)dt.
\]
where $\mu$ and $\bar{\mu}$ denote the Busmann-Hausdorff volume or the Holmes-Thompson volume on $(M,F)$ and $(N,F|_N)$, respectively and $\xi_0:=\min_{\xi\in \mathcal {V}^*N\backslash 0}H_\xi$ and $\zeta(\xi_0)$ is the first positive zero of $\left(\mathfrak{s}'_\delta-\frac{H_{\xi_0}}{k}\mathfrak{s}_\delta\right)(t)$.
\end{theorem}

Let $\pi:TM\rightarrow M$ be the tangent bundle over $M$. Unlike Riemannian case, all the connections in Finsler geometry are defined on $\pi^*TM$. Given a local coordinate system $(x^i,y^i)$ of $TM$, let $\nabla$ (resp. $\Gamma^i_{jk}$) denote the Chern connection (resp. connection coefficients), i.e.,
\[\Gamma^k_{ij}(x,y)\frac{\partial}{\partial x^k}:=\nabla_{\frac{\delta}{\delta x^i}|_{(x,y)}}\frac{\partial}{\partial x^j}=:\nabla^y_{\frac{\partial}{\partial x^i}}\frac{\partial}{\partial x^j}, \forall (x,y)\in TM\backslash0.
\]
In general, $\Gamma^k_{ij}(x,y)\neq \Gamma^k_{ij}(x,z)$ if $y\neq z$.
Shen in \cite{Sh1} introduce the {\it T-curvature} to measure the difference of Chern connection coefficients between two points in $TM\backslash0$. More precisely, the T-curvature
$\mathbf{T}:TM\backslash0\times TM\backslash0\rightarrow \mathbb{R}$ is defined by
\begin{equation*}
\mathbf{T}_y(v):=g_y(\nabla^V_vV,y)-g_y(\nabla_v^YV,y), \ \ v\in
T_xM,
\end{equation*}
where $Y$ is a geodesic field such that $Y|_x=y$ and $V$ is any extension of $v$. $\mathbf{T}=0$ if and only if $(M,F)$ is Berwald space. In this case, $(M,F)$ is modeled on a
single Minkowski space and $\Gamma^i_{jk}$ coincide with some Riemannian metric's Christoffel symbols (cf. \cite{BCS,Sh1,Sz}).

We now consider the special case when $N=\gamma$ is a geodesic. It should be remark that neither the Busemann-Hausdorff volume or the Holmes-Thompson volume of $\gamma$ is equal to the length of $\gamma$ unless $F|_\gamma$ is reversible. But we still have the following estimate
\begin{corollary}Let $(M,F)$ be a closed Finsler $m$-manifold with uniform constant $\Lambda_F$ and diameter $d$.
If $\mathbf{K}\geq \delta$ and $\mathbf{T}\leq l$, then for any simple closed geodesic $\gamma$ in $M$,
\[
L_F(\gamma)\geq\frac{\mu(M)}{c_{m-2}\Lambda_F^{(3m+1)/2}\left[\frac{\mathfrak{s}^{m-1}_\delta\left({\min\left\{d,\frac{\pi}{2\sqrt{\delta}}\right\}}\right)}{m-1}+l\int_0^{d}\mathfrak{s}^{m-1}_\delta(t) dt\right]},
\]
where $\mu(M)$ is either the Busemann-Hausdorff volume or the Holmes-Thompson volume of $M$ and $L_F(\gamma)$ is the length of $\gamma$.
\end{corollary}

According to \cite[Lemma 12.2.5]{Sh1}, Klingenberg lemma\cite{Kl} can be extended to the case of a reversible
Finsler metric. This together with the corollary above furnishes
\begin{corollary}
Let $(M,F)$ be a closed reversible Finsler $m$-manifold with uniform constant $\leq\Lambda$, diameter $\leq d$, $\mu(M)\geq V$, $|\mathbf{K}|\leq \delta$ and $\mathbf{T}\leq l$, where $\mu(M)$ is either the Busemann-Hausdorff volume or the Holmes-Thompson volume of $M$. Then
\[
\mathfrak{i}_M\geq \min \left\{ \frac{\pi}{\sqrt{\delta}},\ \frac{V}{2 c_{m-2}\Lambda^{(3m+1)/2}\left[\frac{\mathfrak{s}^{m-1}_{-\delta}(d)}{m-1}+l  \int_0^d\mathfrak{s}^{m-1}_{-\delta}(t) dt\right]}   \right\}.
\]

\end{corollary}

Randers metrics are natural and important Finsler metrics which are defined
as the sum of a Riemannian metric and a 1-form.

\begin{theorem}
Let $(M,F)$ be a compact Randers manifold with $\mathbf{K}\geq \delta$ and let $\gamma$ be a closed geodesic in $M$. Set $b:=\sup_{x\in M}\|\beta\|_{\alpha}$ and $b_1:=\sup_{x\in M}\|\nabla\beta\|_{\alpha}$. Then
\[
L_F(\gamma)\geq \frac{(1-b)^{\frac{m+2}{2}}}{c_{m-2}(1+b)^{\frac{1}{2}}\mathfrak{S}(b,b_1,\delta,d,m)}\max\left\{\frac{\mu_{BH}(M)}{(1+b)^{\frac{m+1}{2}}},(1-b)^{\frac{m+1}{2}}\vol_\alpha(M)\right\},
\]
where
\[
\mathfrak{S}(b,b_1,\delta,d,m)=\frac{\mathfrak{s}_\delta^{m-1}\left(\min\left\{d,\frac{\pi}{2\sqrt{\delta}}\right\}\right)}{m-1}+\frac{b_1(2b^3+5b^2-2b+7)}{2(1-b)^3}\int_0^{d}\mathfrak{s}^{m-1}_\delta(t) dt,
\]
and $\vol_\alpha$ is the Riemannian volume of $M$ induced by $\alpha$.
\end{theorem}

However, Finsler geometry is much more complicated
than Riemannian geometry.

\begin{example}[\cite{BCS}]
Let $M:=\mathbb{S}^2\times \mathbb{S}$. Let $\alpha$ be the canonical Riemannian product metric on $M$, that is,
$\alpha=\sqrt{dr\otimes dr+\sin^2 (r) d\theta \otimes d\theta+dt\otimes dt}$, where $(r,\theta)$ (resp. $t$) is the usual spherical coordinates on $\mathbb{S}^2$ (resp. $\mathbb{S}$). Choose a $1$-form $\beta_\epsilon:=\epsilon\, dt$, where $\epsilon\in [0,1)$. Note that $\beta_\epsilon$ is globally defined on $M$, even though the coordiante $t$ is not. Take $F_\epsilon=\alpha+\beta_\epsilon$. A direct calculation shows that $(M,F_\epsilon)$ satisfy
\[
\mu_{HT}(M)=8\pi^2,\ \mathbf{K}\geq 0,\ \diam(M)\leq 6\pi,
\]
for all $\epsilon\in [0,1)$. Since $F_\epsilon$ is a Berwald metric, $\gamma(s)=(0,0,-s)$ is a closed geodesic of $F_\epsilon$. Clearly, $L_{F_\epsilon}(\gamma)=\pi (1-\epsilon)\rightarrow 0$ as $\epsilon\rightarrow 1$.

\end{example}

\section{Preliminaries}
In this section, we recall some definitions and properties cencerned
with Finsler manifolds. See \cite{BCS,Sh2} for more details.

Let $(M,F)$ be a (connected) Finsler manifold with Finsler metric
$F:TM\rightarrow [0,\infty)$. Define $S_xM:=\{y\in T_xM:F(x,y)=1\}$
and $SM:=\underset{x\in M}{\cup}S_xM$. Let $(x,y)=(x^i,y^i)$ be
local coordinates on $TM$, and let $\pi :TM\rightarrow M$ and
$\pi_1:SM\rightarrow M$ be the natural projections. Denote by
$c_{n-1}$ the volume of the Euclidean unit $(n-1)$-sphere. Define
\begin{align*}
&\ell^i:=\frac{y^i}{F},\ g_{ij}(x,y):=\frac12\frac{\partial^2
F^2(x,y)}{\partial y^i\partial
y^j},&A_{ijk}(x,y):=\frac{F}{4}\frac{\partial^3 F^2(x,y)}{\partial
y^i\partial y^j\partial y^k},\\
&\gamma^i_{jk}:=\frac12 g^{il}\left(\frac{\partial g_{jl}}{\partial
x^k}+\frac{\partial g_{kl}}{\partial x^j}-\frac{\partial
g_{jk}}{\partial x^l}\right),
&N^i_j:=\left(\gamma^i_{jk}\ell^j-A^i_{jk}\gamma^k_{rs}\ell^r
\ell^s\right)\cdot F.
\end{align*}
The Chern connection $\nabla$ is defined on the pulled-back bundle
$\pi^*TM$ and its forms are characterized by the following structure
equations:

(1) Torsion freeness: $dx^j\wedge\omega^i_j=0$;

(2) Almost $g$-compatibility: $d
g_{ij}-g_{kj}\omega^k_i-g_{ik}\omega^k_j=2\frac{A_{ijk}}{F}(dy^k+N^k_l
dx^l)$.

From above, it's easy to obtain $\omega^i_j=\Gamma^i_{jk}dx^k$, and
$\Gamma^i_{jk}=\Gamma^i_{kj}$.

The curvature form of the Chern connection is defined as
\[
\Omega^i_j:=d\omega^i_j-\omega^k_j\wedge\omega^i_k=:\frac{1}{2}R^i_{j\,kl}d
x^k\wedge d x^l+P^i_{j\,kl}d x^k\wedge\frac{d y^l+N^l_s dx^s}{F}.
\]
Given a non-zero vector $V\in T_xM$, the flag curvature $K(y,V)$ on
$(x,y)\in TM\backslash 0$ is defined as
\[
\mathbf{K}(y,V):=\frac{V^i y^jR_{jikl}y^l
V^k}{g_y(y,y)g_y(V,V)-[g_y(y,V)]^2},
\]
where $R_{jikl}:=g_{is}R^s_{j\,kl}$. And the Ricci curvature of $y$
is defined by
\[
\mathbf{Ric}(y):=\underset{i}{\sum}\,K(y,e_i),
\]
where $e_1,\ldots, e_n$ is a $g_y$-orthonormal base on $(x,y)\in
TM\backslash0$.

Given $y\in T_xM\backslash0$, extend $y$ to a geodesic field $Y$ is a
neighborhood of $x$ (i.e., $\nabla^Y_YY=0$), and define
$T$-curvature $\mathbf{T}$ as
\begin{equation*}
\mathbf{T}_y(v):=g_y(\nabla^V_vV,y)-g_y(\nabla_v^YV,y), \ \ v\in
T_xM,
\end{equation*}
where $V$ is a vector field with $V_x=v$. In any local coordinates $%
(x^i,y^i) $,
\begin{equation*}
\mathbf{T}_y(v)=y^lg_{kl}(y)\{\Gamma^k_{jm}(v)-\Gamma^k_{jm}(y)\}v^jv^m.
\end{equation*}
We say $\mathbf{T}\leq l$ if
\begin{equation*}
\mathbf{T}_y(v)\leq l\left[\sqrt{g_y(v,v)}-g_y\left(v,\frac{y}{F(y)}\right)%
\right]^2F(y),
\end{equation*}
where $y,v\in TM\backslash0$. Similarly, we define the bound $\mathbf{T}\geq
l$.

Given $y\neq0$, we always use $\gamma_y(t)$ to denote a constant speed geodesic
with $\dot{\gamma}_y(0)=y$.
Given any the volume form $d\mu$ on $M$.
In a local coordinate system $(x^i)$, express
$d\mu=\sigma(x)dx^1\wedge\cdots\wedge dx^n$. For $y\in
T_xM\backslash0$, define the distorsion of $(M,F,d\mu)$ as
\[
\tau(y):=\log \frac{\sqrt{\det(g_{ij}(x,y))}}{\sigma(x)}.
\]
And we define the S-curvature $\mathbf{S}$ as
\[
\mathbf{S}(y):=\frac{d}{dt}[\tau(\dot{\gamma}_y(t))]|_{t=0},
\]
Two volume forms used frequently are the Busemann-Hausdorff volume
form $d\mu_{BH}$ and the Holmes-Thompson volume form $d\mu_{HT}$,
respectively. Given a local coordinate system $(x^i)$,
$d\mu_{BH}=\sigma_{BH}(x)dx^1\wedge\cdots\wedge dx^n$ and
$d\mu_{HT}=\sigma_{HT}(x)dx^1\wedge\cdots\wedge dx^n$, where
\begin{align*}
\sigma_{BH}(x)&=\frac{\vol_{\mathbb{R}^n}(\mathbb{B}^n)}{\vol_{\mathbb{R}^n}\{y\in
T_xM:F(x,y)<1 \}},\\
\sigma_{HT}(x)&=\frac{1}{c_{n-1}}\int_{S_xM}\det
g_{ij}(x,y)\left(\overset{n}{\underset{i=1}{\sum}}(-1)^{i-1}
y^idy^1\wedge\cdots \wedge\hat{dy^i}\wedge \cdots\wedge dy^n\right).
\end{align*}

The reversibility $\lambda_F$ of $(M,F)$ is defined by (\cite{Ra})
\begin{equation*}
\lambda_F:=\underset{(x,y)\in
TM\backslash0}{\sup}\frac{F(x,-y)}{F(x,y)}.
\end{equation*}
Clearly $\lambda_F\geq 1$ and $\lambda_F=1$ if and only if $F$ is
reversible. The uniformity constant of $(M,F)$ is defined by
(\cite{Egl})
\[
\Lambda_F:=\underset{X,Y,Z\in SM}{\sup}\frac{g_X(Y,Y)}{g_Z(Y,Y)}.
\]
Clearly, $\lambda_F\leq\sqrt{\Lambda_F}$ and $\Lambda_F=1$ if and
only if $F$ is Riemannian.

By \cite{Sh1}, the Legendre transformation $\mathcal {L} : TM
\rightarrow T^*M$ is defined by
\[\mathcal {L}(Y)=
\left \{
\begin{array}{lll}
&0, &Y=0,\\
&g_Y(Y,\cdot),&Y\neq 0.
\end{array}
\right.
\]
For any $x\in M$, the Legendre transformation is a smooth
diffeomorphism from $T_xM\backslash\{0\}$ onto
$T^*_xM\backslash\{0\}$.

\section{Conormal bundle}

Throughout this paper, we assume that $(M,F)$ is a forward complete Finsler $m$-manifold and $i:N\hookrightarrow M^m$ is a connected $k$-dimensional submanifold of $M$, $0\leq k\leq m$. The rules that govern our index gymnastics are as follows: $i,j$ run from $1$ to $m$. $\alpha,\beta$ run from $1$ to $k$. $A,B$ run from $k+1$ to $m$. $\mathfrak{g}, \mathfrak{h}$ run from $k+1$ to $m-1$.
And $\mathbbm{a},\mathbbm{b}$ run from $1$ to $m-1$.

According to \cite{Ru,Sh1}, the "normal bundle" $\mathcal {V}N$ of $N$ is defined
as
\[
\mathcal {V}N:=\{n\in TM: n=0 \text{ or } g_{n}(n, TN)=0\}.
\]
It is remarkable that if $k\geq 2$, then $\mathcal {V}N$ is not a
vector bundle unless $F$ is Riemannian. Even in the case that $k=1$, $\mathcal
{V}N$  maybe not a vector bundle unless $F$ is reversible.

Consider the following subbundle of $T^*M$
\[
\mathcal {V}^*N:=\{\omega\in T^*M: \,i^*\omega=0\}.
\]
It is easy to see that $\mathcal
{V}^*N=\mathcal {L}(\mathcal {V}N)$, where $\mathcal {L}:TM\rightarrow T^*M$ is the Legendre
transformation. Note that $\mathcal {L}$ is a
homeomorphism from $TM$ to $T^*M$ and a diffeomorphism from
$TM\backslash0$ to $T^*M\backslash0$. Hence, $\mathcal
{V}^*N$ is called the {\it conormal bundle} over $N$ in $M$.
\begin{example}
Let $F(y)=\alpha(y)+\beta(y)$ be a Randers norm on a vector space $V$, where $\alpha$ is an Euclidean norm and $\beta$ is a $1$-form. Let $N=\{vt+w:t\in \mathbb{R}\}$ be a straight line in $V$, where $v\neq0$ and $w$ are constant vectors in $V$. Clearly, $\mathcal {V}^*N=\{\xi:\xi(v)=0\}$ is a bundle. However,
a direct calculation shows that
$\mathcal {V}N=\{n:\, \alpha(n)\beta(v)=-\langle v, n\rangle \}$, where $\langle \cdot, \cdot\rangle$ is the inner product induced by $\alpha$. Hence, $\mathcal {V}N$ is a vector bundle if and only if $\beta=0$, i.e., $F=\alpha$.
\end{example}

Let $\pi:\mathcal {V}^*N\rightarrow N$ denote the bundle projection. Given $\eta\in \mathcal {V}^*N$, set $x=\pi(\eta)$. There exist local coordinate systems $(U_N,u^\alpha)$ and $(U_M,x^i)$ of $x$ and $i(x)$, respectively, such that $U_N\subset U_M$,
$u^\alpha=x^\alpha$ and $x^A|_{U_N}=0$. Hence, for each
$\xi\in\pi^{-1}(U_N)$, $\xi=\xi_A dx^A$ and therefore, $\pi^{-1}(U_N)\approx U_N\times \mathbb{R}^{m-k}$. We call $(u^\alpha,
\xi_A)$ the (local) {\it canonical coordinates} on $\mathcal
{V}^*N$.

For simplicity, set $\mathcal {V}_x^*N:=\pi^{-1}(x)$. And we always identify $N$ with the zero section of $\mathcal
{V}^*N$.

\begin{definition}
 Given a
point $x\in N$ and $\xi\in \mathcal
{V}_x^*N\backslash 0$, the co-second fundamental form of
$N$ along $\xi$ in $M$ is defined as
\[
h_\xi(X,Y):=\langle \xi,
\nabla^{n}_X\overline{Y}\rangle=g_{n}(n,\nabla^{n}_X\overline{Y}),\ \forall \, X,Y\in T_xN,
\]
where $n:=\mathcal {L}^{-1}(\xi)$ and $\overline{Y}$ is any
extension of $Y$ to a tangent vector field on $N$.
\end{definition}
By a direct calculation, one can check that
$h$ is well-defined and $h_\xi:T_xM\otimes T_xM\rightarrow \mathbb{R}$ is a symmetric bilinear form. Let $\Lambda^\xi$ denote the normal curvature defined in \cite{Sh5}. Then $h_\xi(X,X)=-\Lambda^\xi(X)-\mathbf{T}_n(X)$.

\begin{definition}
Given any $\xi\in\mathcal {V}^*N\backslash 0$, co-Weingarten map
$\mathfrak{A}^\xi:T_xN\rightarrow T_xN$ is defined as
\[
\mathfrak{A}^\xi(X):=-(\nabla_X^{\bar{n}}\bar{n})^{\top_n},
\]
where $\bar{n}=\mathcal {L}^{-1}(\bar{\xi})$, $\bar{\xi}$ is an
extension of $\xi$ to a co-normal vector field on $N$, and the
superscript $\top_n$ denotes projection to $T_xN$ by $g_n$.
\end{definition}

\begin{proposition}
$\mathfrak{A}^\xi$ is well-defined and
\[
g_n(Y,\mathfrak{A}^\xi(X))=h_\xi(X,Y), \forall \,X,Y\in T_xM,
\]
where $n=\mathcal {L}^{-1}(\xi)$.
\end{proposition}
\begin{proof}
Choose a local canonical coordinates system $(u^\alpha,\xi_A)$ around $\xi$. Let
$\bar{n}$ be defined above. Thus,
\begin{align*}
\bar{n}=g^{*Aj}_{(\bar{\xi})}\,\bar{\xi}_A\frac{\partial}{\partial
x^j}=:Z^j\frac{\partial}{\partial x^j},\
\nabla_X^{\bar{n}}\bar{n}=\left[X^\alpha\frac{\partial Z^j}{\partial
x^\alpha}+\Gamma^j_{i\alpha}(n)Z^iX^\alpha\right]\frac{\partial}{\partial
x^j},\tag{2.1}
\end{align*}
and
\[
\frac{\partial Z^j}{\partial x^\alpha}=\bar{\xi}_A\frac{\partial
g^{*Aj}}{\partial
x^\alpha}(\bar{\xi})+g^{*Aj}_{(\bar{\xi})}\,\frac{\partial
\bar{\xi}_A}{\partial x^\alpha}.
\]
Hence,
\[
g_n\left(X^\alpha\frac{\partial Z^j}{\partial x^\alpha}\frac{\partial }{\partial x^j},
\frac{\partial}{\partial x^\beta}\right)=g^*_{\beta
j}(\xi)\left(\frac{\partial g^{*Aj}}{\partial
x^\alpha}({\xi})\right){\xi}_A X^\alpha.\tag{2.2}
\]
It follows from (2.1) and (2.2) that $\mathfrak{A}^\xi$ is
well-defined.

Given $X,Y\in T_xN$. Extend $Y$ to a tangent vector field
$\overline{Y}$ on $N$. Then
\[
g_{n}(\mathfrak{A}^\xi(X),Y)=g_{\bar{n}}(-\nabla_X^{\bar{n}}\bar{n},\overline{Y})
=g_{n}(n,\nabla^n_X\overline{Y})=h_\xi(X,Y).
\]\end{proof}

Given $\xi\in \mathcal {V}^*_xN\backslash 0$.
Note that $i^*g_{n}$ is a Euclidean metric on $T_xN$, where $n={\mathcal {L}^{-1}(\xi)}$. We define the {\it co-mean curvature} of $N$ along $\xi$ by
\[
H_\xi:=\text{tr}_{i^*g_{n}}h_\xi.
\]
It is easy to see that $H_\xi=\sum_{\alpha}\lambda_\alpha$, where $\lambda_\alpha$ is the eigenvalues of $\mathfrak{A}^\xi$.
In the Riemannian case, $h_\xi$ is the second fundamental form $h_n$ and $\mathfrak{A}^\xi$ is the Weingarten map $\mathfrak{A}^n$.

Given $\xi\in \mathcal {V}^*_xN$ with $F^*(\xi)=1$, let $n:=\mathcal {L}^{-1}(\xi)$, $T=\dot{\gamma}_n(t)$, $n^\bot=\{X\in T_xM:
g_n(n,X)=0\}$ and $T^\bot_xN=\{X\in T_xM: g_n(X,T_xN)=0\}$.
The
collection $\mathfrak{T}$ of transverse Jacobi fields along the geodesic
$\gamma_n(t)$, $t\in [0,a]$, is defined by
\[
\mathfrak{T}:=\{J:J \text{ is a Jacobi field},\ g_T(T,J)=0,\ J(0)\in
T_xN, \ (\nabla_T^TJ)(0)+\mathfrak{A}^\xi(J(0))\in T^\bot_xN\}.
\]
It is easy to see that $\mathfrak{T}$ is a vector space and
$(\nabla^T_TJ)(0)\in n^\bot$. A similar argument to the one given in \cite[p. 141]{IC}
shows that $\dim (\mathfrak{T})=m-1$.

\begin{example}Let $N$, $M$ and $\gamma_n$ be as above. If the
flag curvature $\mathbf{K}_T(T;\cdot)=k$ and $
\mathfrak{A}^\xi=\lambda\cdot \id$, then the transverse Jacobi field $J$ has
the form
\[
J(t)=[\mathfrak{s}'_k-\lambda\mathfrak{s}_k](t)E(t)+\mathfrak{s}_k(t)F(t),
\]
where $E(t)$ and $F(t)$ are two
parallel vector fields along $\gamma_n$ such that $E(0)\in T_xN$ and
$F(0)\in T^\bot_xN\cap n^\bot$.\end{example}

Let $\mathfrak{X}$ denote the collection of all vector fields $X$ along $\gamma_n$ such that $g_T(T,X)=0$ and $X(0)\in T_xN$ and let $\mathfrak{X}_0$
consist of those elements of $\mathfrak{X}$ that vanish at $t = a$. On $\mathfrak{X}$, the index is defined by
\[
I(X,Y):=-h_{\xi}(X(0),Y(0))+\int_0^ag_T(\nabla^T_TX,\nabla^T_TY)+R_T(T,X,T,Y)\,dt.
\]
\begin{definition}
Let $N,M$ and $\gamma_n$ be as above. A point $\gamma_n(t)$ is said to be focal to $N$ along $\gamma$ if there exists a nontrivial transverse Jacobi field
$J$ such that $J(t) = 0$.
\end{definition}
Then we have the following lemma.
\begin{lemma}
Given any $X\in \mathfrak{X}_0$. If $\gamma_n(t)$ has not focal points along
$\gamma_n$ on $(0, a]$ to $N$, then $I(X,X)\geq 0$ with equality if and only if $X=0$.
\end{lemma}
\begin{proof}From assumption, there exists $n-1$ transverse Jacobi fields $J_i$ such that $\{T,J_i\}$ is a frame field along $\gamma_n$. We can suppose that $X(t)=f^i(t)J_i(t)$ where $f^i(a)=0$. Set $A:=(f^i)'\cdot J_i$ and $B:=f^i\cdot\nabla^T_T J_i$. Then $g_T(\nabla^T_TX,\nabla^T_TX)=g_T(A,A)+g_T(B,B)+2g_T(A,B)$. Using the Jacobi equation, one can easily check that (also see \cite[p. 180]{BCS})
\begin{align*}
&g_T(B,B)+R_T(T,X,T,X)\\
=&\frac{d}{dt}\left[f^if^jg_T(\nabla^T_TJ_i,J_j)\right]-(f^i)' f^jg_T(\nabla^T_TJ_i,J_j)-f^i (f^j)'g_T(\nabla^T_TJ_i,J_j).\tag{2.3}
\end{align*}
The Lagrange identity (\cite[p. 135]{BCS}) yields
\[
g_T(\nabla^T_TJ_i,J_j)-g_T(J_i,\nabla^T_TJ_j)\equiv g_T(\nabla^T_TJ_i,J_j)(0)-g_T(J_i,\nabla^T_TJ_j)(0).
\]
Recall that $(\nabla^T_TJ_j)(0)+\mathfrak{A}^\xi(J_j(0))\in T^\bot_xN$ and $T(0)=n$. Hence,
\[
g_T(J_i(0),(\nabla^T_TJ_j)(0))=g_T(J_i(0),-\mathfrak{A}^\xi(J_j(0)))=-h_\xi(J_i(0),J_j(0)),\tag{2.4}
\]
which implies that $g_T(\nabla^T_TJ_i,J_j)=g_T(J_i,\nabla^T_TJ_j)$. Thus, from (2.3), (2.4) and $f^i(a)=0$, we have
\[
I(X,X)=\int^a_0g_T(A,A)dt\geq0,
\]
with equality $A=0$, i.e., $X=0$.
\end{proof}

Using the lemma above, it is not hard to show
\begin{theorem}
Suppose that $\gamma_n(t)$ has not focal points along
$\gamma_n$ on $(0, a]$ to $N$. Given $X\in \mathfrak{X}$, let $J$ denote the unique transverse Jacobi field along $\gamma_n$
such that $J(a)=X(a)$. Then $I(X,X)\geq I(J,J)$ with equality if and only if $X=J$.
\end{theorem}

The proof of the following theorem is almost the same as the one of \cite[Propositon 7.4.1]{BCS}. Hence, we omit it here.
\begin{theorem}
Suppose that some point $\gamma_n(t_0)$, $0<t_0<a$ is focal to $N$ along $\gamma_n$. Then there is $U\in \mathfrak{X}_0$ such that $I(U,U)<0$.
\end{theorem}

\begin{remark}
Given an arbitrary point $p\in M\backslash N$, there exists a unit speed minimizing geodesic $\gamma_n$ from $N$ to $p$. A simple first variation argument yields $\xi:=\mathcal {L}(n)\in \mathcal {V}^*N$. If $N$ has a focal point $\gamma_n(t_0)$ along $\gamma_n$, then $d(N,p)\leq t_0$ follows from Lemma 3.5, Theorem 3.7 and the second variation of arc length formula.
\end{remark}

\section{Conormal exponential map}
We define the conormal exponential map $\Exp:\mathcal {V}^*N\rightarrow M$ by
\[
\Exp(\xi):=\exp_{\pi(\xi)}(\mathcal {L}^{-1}(\xi)).
\]
Let $S^*M:=\{\omega\in T^*M:\,F^*(\omega)=0\}$ and $\mathcal {V}^*SN:=S^*M\cap \mathcal {V}^*N$. Now we have the following
\begin{theorem}
For each $\eta \in \mathcal {V}^*SN$, there exists a small $\epsilon(\eta)>0$ and an open neighborhood $\mathcal {W}$ of $\eta$ in $\mathcal {V}^*SN$ such that $\Exp_{*t\xi}$ is nonsingular for all $\xi\in \mathcal {W}$ and $t\in (0, \epsilon(\eta))$. In particular, $\Exp$ is $C^1$ on $N\subset \mathcal {V}^*N$ if and only if $F$ is Riemannian.
\end{theorem}
\begin{proof}For the sake of clarity, we use $(x,\xi)$ to denote a point $\xi\in \mathcal {V}^*N$. Given $(x_0,\eta_0)\in \mathcal {V}^*SN\subset \mathcal {V}^*N$. Let $(U_N,u^\alpha)$ and $(U_M, x^i)$ be two local coordinate systems around $x_0$ and $i(x_0)$, respectively, such that $x^\alpha|_{U_N}=u^\alpha$ and $x^A|_{U_N}=0$. We can choose a small $\delta>0$ such that $\Exp(\mathcal {D})\subset U_M$, where $\mathcal {D}=\{(x,t\eta):t\in [0,\delta),\, (x,\eta)\in \mathcal {V}^*SU_N\}$. Let $(x^i,y^i)$ and $(u^\alpha,\xi_A)$ be local canonical coordinates on $TM$ and $\mathcal {V}^*N$, respectively.
For each $(x,t\eta)\in \mathcal {D}$, we have
\[
\Exp_{*(x,t\eta)}\frac{\partial}{\partial u^\alpha}=\left.\frac{\partial \exp(x,\mathcal {L}^{-1}(\xi))}{\partial u^\alpha}\right|_{x,\,\xi=t\eta}
=\left[\delta^i_\alpha+H(t,x,\eta)^i_\alpha\right]\frac{\partial}{\partial x^i},
\]
where
\[
H(t,x,\eta)^i_\alpha:=\left[\frac{\partial \exp^i}{\partial x^\alpha}(x,t\mathcal {L}^{-1}(\eta))-\delta^i_\alpha\right]+\frac{\partial\exp^i}{\partial y^k}(x,t\mathcal {L}^{-1}(\eta))\cdot\frac{\partial g^{*Ak}}{\partial u^\alpha}(x,\eta)\cdot t\eta_A.
\]
Likewise,
\[
\Exp_{*(x,t\eta)}\frac{\partial}{\partial \xi_A}=g^{*Ak}(\eta)\left[\delta^i_k+L(t,x,\eta)^i_k\right]\frac{\partial}{\partial x^i},\tag{2.5}
\]
where
\[
L(t,x,\eta)^i_k:=\frac{\partial\exp^i}{\partial y^k}(x,t\mathcal {L}^{-1}(\eta))-\delta^i_k.
\]
Clearly, $\underset{t\rightarrow 0^+}{\lim}H(t,x,\eta)^i_\alpha=\underset{t\rightarrow 0^+}{\lim}L(t,x,\eta)^i_k=0$, which together with (2.5) implies that $\Exp$ is $C^1$ on $N\subset \mathcal {V}^*N$ if and only if $F$ is Riemannian.
From above, the matrix of $\Exp_{*(x,t\eta)}$ is
\begin{equation*}
S(t,x,\eta)=\left(
  \begin{array}{cc}
   \delta_\alpha^\beta+H(t,x,\eta)^\beta_\alpha  &\ H(t,x,\eta)^A_\alpha\\
    g^{*A\alpha}_{(\eta)}+g^{*Ak}_{(\eta)}L(t,x,\eta)_k^\alpha  &\ g^{*AB}_{(\eta)}+g^{*Ak}_{(\eta)}L(t,x,\eta)_k^B\\
  \end{array}
\right).
\end{equation*}
Since $\det S(0,x_0,\eta_0)>0$, there exists a small $\epsilon(x_0,\eta_0)>0$ and an open neighborhood $\mathcal {W}$ of $(x_0,\eta_0)$ in $\mathcal {V}^*SN$ such that $\Exp_{*t\xi}$ is nonsingular for all $\xi\in \mathcal {W}$ and $t\in (0, \epsilon(x_0,\eta_0))$.
\end{proof}

Let $\pi_1:\mathcal {V}^*SN\rightarrow N$ be the natural projection. Thus, for each $x\in N$, $\pi_1^{-1}(x):=\mathcal {V}^*_xSN$ is a $(n-k-1)$-dimensional Minkowski unit sphere in $T^*_xM$. Given a local coordinate system $(u^\alpha,\theta_{\mathfrak{g}})$ on $\mathcal {V}^*SN$, where $(u^\alpha)$ are local coordinates of $N$, and for fixed $x=(u^\alpha)$, $(\theta_{\mathfrak{g}})$ are the local coordinates of $\mathcal {V}^*_xSN$. Hence, we obtain a local {\it cone coordinate system} $(t,u^\alpha,\theta_\mathfrak{g})$ on $\mathcal {V}^*N\backslash N$, that is, for $\xi\in \mathcal {V}^*N\backslash N$, $t=F^*(\xi)$ and $\xi/F^*(\xi)=(u^\alpha,\theta_\mathfrak{g})$.

Define a map $\E:[0,+\infty)\times \mathcal {V}^*SN\rightarrow M$ by $\E(t,\xi)=\Exp(t\xi)$.
It is easy to see that
\[
\E_{*(t,\xi)}\frac{\partial}{\partial t}=\left(\exp_{\pi(\xi)}\right)_{*t\mathcal {L}^{-1}(\xi)}\mathcal {L}^{-1}(\xi).
\]
In particular, $\E_{*(0,\xi)}=\mathcal {L}^{-1}(\xi)$.

\begin{proposition}
\[
J(t)=\E_{*(t,\xi)}\frac{\partial}{\partial u^\alpha}
\]
is a transverse Jacobi field along $\gamma_{\mathcal {L}^{-1}(\xi)}(t)$ such that $J(0)=\frac{\partial}{\partial u^\alpha}$ and $(\nabla^T_TJ)(0)+\mathfrak{A}^\xi(J(0))\in T^\bot_xN$, where $T:=\dot{\gamma}_{\mathcal {L}^{-1}(\xi)}(t)$ and $x:=\pi(\xi)$.
\end{proposition}
\begin{proof} Suppose $\xi=(u^\beta,\theta_\mathfrak{g})$.
Set $\gamma(s)=(u^\beta(s))$ and $\xi(s)=(u^\beta(s),\theta_\mathfrak{g})$, where $u^\beta(s)=u^\beta+s\cdot\delta^\beta_\alpha$, $s\in (-\epsilon,\epsilon)$. Consider the variation $\sigma(t,s)=\E(t,\xi(s))=\exp_{\gamma(s)}t\mathcal {L}^{-1}(\xi(s))$. Thus,
\[
J(t)=\left.\frac{\partial}{\partial s}\right|_{s=0}\sigma(t,s)=\E_{*(t,\xi)}\frac{\partial}{\partial u^\alpha}
\]
is a Jacobi field along $\gamma_{\mathcal {L}^{-1}(\xi)}(t)$. And $J(0)=\left.\frac{\partial}{\partial s}\right|_{s=0}\sigma(0,s)=\left.\frac{\partial}{\partial s}\right|_{s=0}\gamma(s)=\frac{\partial}{\partial u^\alpha}$.
Set $\mathcal {T}(t,s):=\frac{\partial}{\partial t}\sigma(t,s)$. Clearly, $\mathcal {T}(t,0)=T(t)$.
And we have
\[
(\nabla^T_TJ)(0)=\nabla^{\mathcal {L}^{-1}(\xi)}_{J(0)}\mathcal {T}(0,s)=\nabla^{\mathcal {L}^{-1}(\xi)}_{J(0)}\mathcal {L}^{-1}(\xi(s)),
\]
which implies $(\nabla^T_TJ)(0)-\mathfrak{A}^{\xi}(J(0))\in T^\bot_xN$. Since $F^*(\xi(s))=1$,
\[
g_T(T,(\nabla^T_TJ)(0))=g_{\mathcal {L}^{-1}(\xi)}(\mathcal {L}^{-1}(\xi),\nabla^{\mathcal {L}^{-1}(\xi)}_{J(0)}\mathcal {L}^{-1}(\xi(s)))=\frac{1}{2}J(0)(F^2(\mathcal {L}^{-1}(\xi(s))))=0.
\]
Hence, $J$ is a transverse Jacobi field along $\gamma_{\mathcal {L}^{-1}(\xi)}(t)$.\end{proof}

\begin{proposition}
\[
J(t)=\E_{*(t,\xi)}\frac{\partial}{\partial \theta_\mathfrak{g}}
\]
is a transverse Jacobi field along $\gamma_{\mathcal {L}^{-1}(\xi)}(t)$ such that $J(0)=0$ and $(\nabla^T_TJ)(0)=\mathcal {L}^{-1}_{*\xi}(\frac{\partial}{\partial\theta_\mathfrak{g}})$, where $T=\dot{\gamma}_{\mathcal {L}^{-1}(\xi)}$, $x=\pi(\xi)$ and $\mathcal {L}^{-1}_{*\xi}:T_\xi(T^*_xM\backslash0)\rightarrow T_{\mathcal {L}^{-1}(\xi)}(T_xM\backslash0)$ is the tangent map.
\end{proposition}
\begin{proof}
Suppose that $\xi=(u^\alpha,\theta_\mathfrak{g})$. Set $\xi(s)=(u^\alpha,\theta_\mathfrak{g}(s))$, $s\in (-\epsilon,\epsilon)$, such that $\frac{d}{ds}|_{s=0}\xi(s)=\frac{\partial}{\partial \theta_\mathfrak{g}}$. Consider the variation
$\sigma(t,s)=\E(t,\xi(s))=\exp_{\pi(\xi)}t\mathcal {L}^{-1}(\xi(s))$. Clearly,
\[
J(t)=\E_{*(t,\xi)}\frac{\partial}{\partial \theta_\mathfrak{g}}=\left.\frac{\partial}{\partial s}\right|_{s=0}\sigma(t,s)=\left(\exp_{\pi(\xi)}\right)_{*t\mathcal {L}^{-1}(\xi)}t\mathcal {L}^{-1}_{*\xi}\left(\frac{\partial}{\partial \theta_\mathfrak{g}}\right).
\]
Since $F(\mathcal {L}^{-1}(\xi(s)))=1$,
\[
0=\left.\frac{d}{ds}\right|_{s=0}F^2(\mathcal {L}^{-1}(\xi(s)))=2g_{\mathcal {L}^{-1}(\xi)}\left(\mathcal {L}^{-1}(\xi),\mathcal {L}^{-1}_{*\xi}\left(\frac{\partial}{\partial \theta_\mathfrak{g}}\right)\right).
\]
By the Guass lemma, we have
\[
g_T(T,J)=g_{\mathcal {L}^{-1}(\xi)}\left(\mathcal {L}^{-1}(\xi),t\mathcal {L}^{-1}_{*\xi}\left(\frac{\partial}{\partial \theta_\mathfrak{g}}\right)\right)=0.
\]
\end{proof}

Note that $T_{(t,\xi)}(\mathcal {V}^*N)=\mathbb{R}\frac{\partial}{\partial t}\oplus T_\xi(\mathcal {V}^*SN)$, for all $t>0$ and $\xi\in \mathcal {V}^*SN$. As a trivial combination of Proposition 4.2 and Proposition 4.3 we get
\begin{lemma}Given $t_0>0$ and $\xi\in \mathcal {V}^*SN$.
If there exists some $X$ such that $\E_{*(t_0,\xi)}X=0$ , then  $X\in T_\xi(\mathcal {V}^*SN)$.
\end{lemma}

Let $\xi(s)$ defined as in the proof of Proposition 4.3. Then $\mathcal {L}^{-1}(\xi(s))$ is a unit normal vector for all $s$ and therefore,
\[
0=\left.\frac{d}{ds}\right|_{s=0}g_{\mathcal {L}^{-1}(\xi(s))}\left(\mathcal {L}^{-1}(\xi(s)),\frac{\partial}{\partial u^\alpha}\right)=g_{\mathcal {L}^{-1}(\xi)}\left(\mathcal {L}^{-1}_{*\xi}\left(\frac{\partial}{\partial \theta_\mathfrak{g}}\right),\frac{\partial}{\partial u^\alpha}\right).
\]
Hence, for each $\xi\in \mathcal {V}^*_xSN$, we have
\begin{align*}
(T_xM,g_{n})&=\mathbb{R}\cdot n\oplus n^\bot=\mathbb{R}\cdot n\oplus T_xN\oplus (T^\bot_xN\cap n^\bot)\\
&=\mathbb{R}\cdot n\oplus \Sp\left \{\frac{\partial}{\partial u^\alpha}\right \} \oplus\Sp\left \{\mathcal {L}^{-1}_{*\xi}\left(\frac{\partial}{\partial \theta_\mathfrak{g}}\right)\right \},
\end{align*}
where $n:=\mathcal {L}^{-1}(\xi)$. For convenience, set $e_\alpha:=\frac{\partial}{\partial u^\alpha}$ and $e_\mathfrak{g}:= \mathcal {L}^{-1}_{*\xi}\left(\frac{\partial}{\partial \theta_\mathfrak{g}}\right)$.

Given $\xi\in \mathcal {V}^*_xSN$, let $n$, $e_\alpha$, $e_\mathfrak{g}$ be defined as above.
Denote by $P_{t;n}$ the parallel translation along
$\gamma_n$ from $T_{\gamma_n(0)}M$ to $T_{\gamma_n(t)}M$ (with
respect to the Chern connection) for all $t\geq 0$. Set $T=\dot{\gamma}_n(t)$. Let
$R_T:=R_T(\cdot,T)T$ and
\[
\mathcal {R}(t,n ):=P^{-1}_{t;n}\circ R_T \circ P_{t;n}:
n^{\bot}\rightarrow n^{\bot}.\]
Let $\mathcal {A}(t,n)$ be the solution of the matrix (or linear
transformation) ordinary differential equation on $n^\bot$:
\[
\left \{
\begin{array}{lll}
&\mathcal {A}{''}+\mathcal {R}(t,y)\mathcal {A}=0,\\
&\mathcal {A}(0,n)e_\alpha=e_\alpha,\ \mathcal {A}'(0,n)e_\alpha=(\nabla^T_TJ_\alpha)(0)\\
 &\mathcal {A}(0,n)e_\mathfrak{g}=0,\ \mathcal {A}'(0,n)e_\mathfrak{g}=e_\mathfrak{g}.
\end{array}
\right.
\]
where $\mathcal {A}'=\frac{d}{dt}\mathcal {A}$ and $J_\alpha(t)=\E_{*(t,\xi)}e_\alpha$. Note that $\gamma_n(t)=P_{t;n}n$. Thus, for each $X\in n^\bot$,
\[
g_{P_{t;n}n}(P_{t;n}n,P_{t;n}\mathcal {A}(t,n)X)=g_n(n,\mathcal {A}(t,n)X)=0,
\]
that is, $P_{t;n}\mathcal {A}(t,n)X$ is a transverse Jacobi filed along $\gamma_n$. In particular, $\E_{*(t,\xi)}e_\alpha=P_{t;n}\mathcal {A}(t,n)e_\alpha$ and  $\E_{*(t,\xi)}\mathcal {L}_{*\xi}\,e_\mathfrak{g}=P_{t;n}\mathcal {A}(t,n)e_\mathfrak{g}$. Set $\mathcal {A}e_{\mathbbm{a}}=:\mathcal {A}_{\mathbbm{a}}^{\mathbbm{b}}e_{\mathbbm{b}}$ and $\det \mathcal {A}:=\det \mathcal {A}_{\mathbbm{a}}^{\mathbbm{b}}$, where ${\mathbbm{a}}=\alpha,\mathfrak{g}$. Clearly, $\det \mathcal {A}$ is independent of choice of basis for $n^\bot$. Then we have the following
\begin{proposition}
Given $\xi\in \mathcal {V}^*SN$, let $n=\mathcal {L}^{-1}(\xi)$. The following statements are mutually equivalent:

(1) $\gamma_n(t_0)$, $0<t_0<\infty$ is a focal point of $N$ along $\gamma_n$.

(2) $\Exp_{*t_0\xi}$ is singular.

(3) $\E_{*(t_0,\xi)}$ is singular.

(4) $\det \mathcal {A}(t_0, n)=0$.
\end{proposition}
\begin{proof} From above, we have (3) $\Leftrightarrow$ (4).

Define a map $\mathscr{F}:(0,+\infty)\times \mathcal {V}^*SN\rightarrow \mathcal {V}^*N\backslash N$ by $\mathscr{F}(t,\xi)=t\xi$. Then $\mathscr{F}$ is a diffeomorphism, since $\mathscr{F}^{-1}(\xi)=(F^*(\xi),\xi/F^*(\xi))$. Clearly, $\Exp_*\circ\mathscr{F}_*=\E_*$, which implies (2) $\Leftrightarrow$ (3).

By Theorem 4.1, there exists $\epsilon(\xi)>0$ such that $\E_{*(t,\xi)}$ is nonsingular for $0<t\leq\epsilon(\xi)$. Thus, $\{\E_{*(t,\xi)}e_\alpha,\E_{*(t,\xi)}\mathcal {L}_{*\xi}\,e_\mathfrak{g}\}$ form a basis for the space of the transverse Jacobi fields along $\gamma_n(t)$, $0\leq t\leq\epsilon(\xi)$.

Suppose $\gamma_n(t_0)$ is a focal point. Then there exists a nontrivial transverse Jacobi field $J$ along $\gamma_n$ such that $J(t_0)=0$. From above, $J(t)=C^\alpha\E_{*(t,\xi)}e_\alpha+C^\mathfrak{g}\E_{*(t,\xi)}\mathcal {L}_{*\xi}\,e_\mathfrak{g}$,
for $0\leq t\leq \epsilon(\xi)$ and therefore, for $t\geq 0$. Here, $C^i$'s are constants not all zero. Then $J(t_0)=0$ implies (3). And it follows from Lemma 4.4 that (1) $\Leftrightarrow$ (3).
\end{proof}

By Theorem 4.1 and Proposition 4.5, we give the following definition.
\begin{definition}
Given $\xi\in \mathcal {V}^*SN$, the focal value $c_f(\xi)$ is defined by
\[
c_f(\xi):=\sup\{r>0: \text{ no point }\gamma_{\mathcal {L}^{-1}(\xi)}(t), \,0<t<r \text{ is focal point }\}
\]
\end{definition}

\begin{proposition}
The function $c_f:\mathcal {V}^*SN\rightarrow (0,+\infty]$ is lower semicontinuous. That is,
\[
\underset{\xi\rightarrow \xi_0}{\lim\inf}{\,c_f(\xi)}\geq c_f(\xi_0).
\]
\end{proposition}
\begin{proof}

\textbf{Step 1:} Suppose $\Exp_*$ is nonsingular at $t\xi$, where $t>0$ and $\xi\in \mathcal {V}^*SN$. Then one can obtain a neighborhood $U$ of $\xi$ in $\mathcal {V}^*SN$ and a small $\epsilon>0$ such that

(i) For each $\eta\in U$ and $s\in (t-\epsilon,t+\epsilon)$, $\Exp_*$ is also nonsingular at $s\eta$.

(ii) In a local trivialization of $\mathcal {V}^*SN$, $U$ is the Cartesian product of an open ball (for the position $u^\alpha$) in $\mathbb{R}^k$ with an open "disk" (for the position $\theta_\mathfrak{g}$) on the standard unit sphere $\mathbb{S}^{n-k-1}$.

\textbf{Step 2:} Given $\xi_0\in \mathcal {V}^*SN$ and $0<r<c_f(\xi_0)$. Let ${\varepsilon}(\xi_0)$ and $\mathcal {W}$ be as in Theorem 4.1. If $r\leq {\varepsilon}(\xi_0)/2$, then we take $\epsilon_r={\varepsilon}(\xi_0)/2$ and $\mathcal {U}=\mathcal {W}$.
Now suppose
$r>{\varepsilon}(\xi_0)/2$. For each $t\in [{\varepsilon}(\xi_0)/2,r]$, one has a neighborhood $U_t$ of $\xi$ and a interval $I_t=(t-\epsilon_t,t+\epsilon_t)$ with the properties stated in Step 1. Then one can find finitely many $\{I_{t_s}\}_{s=1}^k$ such that $\cup_{s}I_{t_s}\supset [{\varepsilon}(\xi_0)/2,r]$. Without loss of generality, we suppose that $t_1<\cdots<t_k$ and $t_k=r$ (so $\epsilon_{t_k}=\epsilon_r$). Now let $\mathcal {U}:=\cap_{i}U_{t_i}\cap\mathcal {W}$. From above, it is easy to see that $\Exp_{*(x,t\xi)}$ is not singular for all $t\in (0,r+\epsilon_r)$ and $(x,\xi)\in
\mathcal{U}$, i.e., $c_f(\xi)> r+\epsilon_r$. Clearly, $\lim_{r\rightarrow c_f(\xi_0)}\epsilon_r=0$.
Hence,
$\underset{\xi\rightarrow\xi_0}{\lim\inf}c_f(\xi)\geq r+\epsilon_r$.
We complete the proof by letting $r\rightarrow c_f(\xi_0)$.\end{proof}

Let $\xi,n,T$ and $e_{\mathbbm{a}},\mathbbm{a}=\alpha,\mathfrak{g}$ be as before. Now, we continue to investigate $\det\mathcal {A}(t,n)$. For simplicity, set $J_\mathbbm{a}(t):=P_{t;n}\mathcal {A}(t,n)e_{\mathbbm{a}}$. A direct calculation yields
\begin{align*}
&\det\left[g_T(J_\mathbbm{a}(t),J_\mathbbm{b}(t))\right]=\det\left[g_n(\mathcal {A}(t,n)e_{\mathbbm{a}},\mathcal {A}(t,n)e_\mathbbm{b})\right]\\
=&(\det\mathcal {A})^2\cdot\det g_n(e_\alpha,e_\beta)\cdot\det g_n(e_\mathfrak{g},e_\mathfrak{h})\tag{6}
\end{align*}
By the Lagrange identity and L'Hospital's rule, we have
\begin{align*}
\underset{t\rightarrow 0^+}{\lim}\frac{g_T\left(J_\alpha(t),J_\mathfrak{g}(t)\right)}{t^2}&=\underset{t\rightarrow 0^+}{\lim}\frac{g_T\left(\nabla^T_TJ_\alpha(t),J_\mathfrak{g}(t)\right)}{t}=\underset{t\rightarrow 0^+}{\lim}g_T\left(\nabla^T_TJ_\alpha(t),\nabla^T_TJ_\mathfrak{g}(t)\right)\\
&=g_n\left((\nabla^T_TJ_\alpha)(0),e_\mathfrak{g}\right).
\end{align*}
And it is easy to see that $\lim_{t\rightarrow 0^+}\frac{g_T(J_\mathfrak{g},J_\mathfrak{h})}{t^2}=g_n(e_\mathfrak{g},e_\mathfrak{h})$. Hence, we have
\begin{align*}
&\underset{t\rightarrow 0^+}{\lim}\frac{\det g_T(J_i,J_k)(t)}{t^{2(m-k-1)}}
=\underset{t\rightarrow 0^+}{\lim}\det\left(
\begin{array}{cc}
   g_{T}(J_\alpha,J_\beta) & \  \frac{g_T\left(J_\alpha(t),J_\mathfrak{g}(t)\right)}{t^2}\\
    g_T\left(J_\mathfrak{h}, J_\beta\right)  &\  \frac{g_T\left(J_\mathfrak{g}(t),J_\mathfrak{h}(t)\right)}{t^2}).\\
  \end{array}
\right)\\
=&\det g_n(e_\alpha,e_\beta)\det g_n(e_\mathfrak{g},e_\mathfrak{h}),
\end{align*}
which implies that
\[
\underset{t\rightarrow 0^+}{\lim}\frac{\det \mathcal {A}(t,n)}{t^{m-k-1}}=1.
\]
Moreover, we have the following
\begin{theorem}Given $\xi\in \mathcal {V}^*SN$. Let $n=\mathcal {L}^{-1}(\xi)$ and $H_\xi=\text{tr}_{g^*_\xi}(h_\xi)$. If the flag curvature $\mathbf{K}(\dot{\gamma}_n(t);\cdot)\geq \delta$, then $c_f(\xi)\leq \min\{\zeta,\pi/\sqrt{\delta}\}$ and
\[
\det \mathcal {A}(t,n)\leq \left(\mathfrak{s}'_\delta-\frac{H_\xi}{k}\mathfrak{s}_\delta\right)^k(t)\cdot\mathfrak{s}^{n-k-1}_\delta(t), \text{ for }t\in [0, c_f(\xi)],
\]
where $\zeta$ is the first positive zero of
\[
\left(\mathfrak{s}'_\delta-\frac{H_{\xi}}{k}\mathfrak{s}_\delta\right)^k(t)
\]
(should such a zero exist; otherwise, set $\zeta=+\infty$).
\end{theorem}
\begin{proof}
Fix some positive number $r<c_f(\xi)$. Let $J_\mathbbm{a}:=P_{t;n}\mathcal {A}e_{\mathbbm{a}}$, for $\mathbbm{a}=\alpha,\mathfrak{g}$. For $s\in (0,r)$, by (6), we have
\[
\frac{(\det \mathcal {A})'}{\det \mathcal {A}}(s)=\frac12\frac{\left({\det g_T(J_\mathbbm{a},J_\mathbbm{b})}\right)'}{{\det g_T(J_\mathbbm{a},J_\mathbbm{b})}}(s).
\]
Note that $\{J_\mathbbm{a}(t)\}$ is a basis for the space $\mathfrak{T}$ of transverse Jacobi fields along $\gamma_n(t)$, $0\leq t\leq s$. Let $\{\bar{J}_\mathbbm{a}(t)\}$ be another $n-1$
transverse Jacobi fields such that $\{T(s),\bar{J}_\mathbbm{a}(s)\}$ is a $g_{T}$-orthonormal basis. Then $\{\bar{J}_\mathbbm{a}(t)\}$ is also a basis for $\mathfrak{T}$. It is easy to check that
\[
\frac{(\det \mathcal {A})'}{\det \mathcal {A}}(s)=\frac12\frac{\left({\det g_T(J_\mathbbm{a},J_\mathbbm{b})}\right)'}{{\det g_T(J_\mathbbm{a},J_\mathbbm{b})}}(s)=\frac12\frac{\left({\det g_T(\bar{J}_\mathbbm{a},\bar{J}_\mathbbm{b})}\right)'}{{\det g_T(\bar{J}_\mathbbm{a},\bar{J}_\mathbbm{b})}}(s).\tag{*1}
\]
A direct calculation shows
\begin{align*}
\frac{\left({\det g_T(\bar{J}_\mathbbm{a},\bar{J}_\mathbbm{b})}\right)'}{{\det g_T(\bar{J}_\mathbbm{a},\bar{J}_\mathbbm{b})}}(s)
=2\cdot\underset{\mathbbm{a}}{\sum}\left({ g_T(\nabla^T_T\bar{J}_\mathbbm{a},\bar{J}_\mathbbm{a})}\right)'(s)=2\cdot\underset{\mathbbm{a}}{\sum} I_{[0,s]}(\bar{J}_\mathbbm{a},\bar{J}_\mathbbm{a}),\tag{*2}
\end{align*}
where $I_{[0,s]}$ is the index restricted to $\gamma_n(t)$, $0\leq t\leq s$.

Consider the solution $\mathcal {A}_\delta(t)$ to the
matrix differential equation in $n^\bot$:
\[
\mathcal {A}_\delta''+k\mathcal {A}_\delta=0,
\]
with the same initial conditions as $\mathcal {A}(t)$. Let $\{f_\alpha\}$ be a $g_n$-orthonormal basis for $T_xN$ consisting of eigenvectors
of the Weingarten map $\mathfrak{A}^\xi$, with respective eigenvalues $\lambda_\alpha$. And let $\{f_{\mathfrak{g}}\}$ be an orthonormal basis for $n^\bot\cap T^\bot_xN$. Then we have
\begin{align*}
\mathcal {A}_\delta f_\alpha&=(\mathfrak{s}'_\delta-\lambda_\alpha \mathfrak{s}_\delta)(t)\cdot f_\alpha+C_\alpha^\mathfrak{g} \mathfrak{s}_\delta(t)\cdot f_\mathfrak{g},\\
 \mathcal {A}_\delta f_\mathfrak{g}&=\mathfrak{s}_\delta(t)\cdot f_\mathfrak{g},
\end{align*}
where $C_\alpha^\mathfrak{g}$'s are constants determined by the initial conditions of $\mathcal {A}_k(t)$. It is easy to see that
\[
\det \mathcal {A}_\delta(t)=\mathfrak{s}_\delta^{m-k-1}(t)\cdot\overset{k}{\underset{\alpha=1}{\Pi}}(\mathfrak{s}'_\delta-\lambda_\alpha \mathfrak{s}_\delta)(t).
\]

Suppose $r<\zeta_0$, where $\zeta_0$ is the first positive zero of $\det \mathcal {A}_\delta(t)$. Thus, $\{T,P_{t;n}\mathcal {A}_\delta(t)f_\mathbbm{a}\}$ is a basis at $\gamma_n(t)$ for all $t\in (0,r]$. Hence, one can find constants $C_\mathbbm{a}^\mathbbm{b}$ such that $\det C^\mathbbm{b}_\mathbbm{a}\neq 0$ and $\bar{J}_\mathbbm{a}(s)=C_\mathbbm{a}^\mathbbm{b}P_{s;n}\mathcal {A}_\delta(s)f_\mathbbm{b}$.
Consider the vector fields $Y_\mathbbm{a}(t):=C_\mathbbm{a}^\mathbbm{b}P_{t;n}\mathcal {A}_\delta(t)f_\mathbbm{b}$. Clearly, $\nabla^T_T\nabla^T_TY_\mathbbm{a}+\delta Y_\mathbbm{a}=0$, $Y_\mathbbm{a}(s)=\bar{J}_\mathbbm{a}(s)$ and $g_T(T,Y_\mathbbm{a})=0$. Theorem 3.6 then yields
\begin{align*}
&\underset{\mathbbm{a}}{\sum}I_{[0,s]}(\bar{J}_\mathbbm{a},\bar{J}_\mathbbm{a})\leq \underset{\mathbbm{a}}{\sum}I_{[0,s]}(Y_\mathbbm{a},Y_\mathbbm{a})\\
\leq&\underset{\mathbbm{a}}{\sum}\left(-h_\xi(Y_\mathbbm{a}(0),Y_\mathbbm{a}(0))+\int^s_0g_T(\nabla^T_TY_\mathbbm{a},\nabla^T_TY_\mathbbm{a})-\delta g_T(Y_\mathbbm{a},Y_\mathbbm{a})dt\right)\\
=&\underset{\mathbbm{a}}{\sum}g_T(\nabla^T_TY_\mathbbm{a},Y_\mathbbm{a})(s).\tag{*3}
\end{align*}

Note that $\{l_\mathbbm{a}:=C^\mathbbm{b}_\mathbbm{a}\cdot f_\mathbbm{b}\}$ is also a basis for $n^\bot$. Set $\mathcal {A}_\delta l_\mathbbm{a}=:(\mathcal {A}_\delta)_\mathbbm{a}^\mathbbm{b}\cdot l_\mathbbm{b}$. Since $g_T(Y_\mathbbm{a},Y_\mathbbm{b})(s)=\delta_{\mathbbm{a}\mathbbm{b}}$, $(\mathcal {A}_\delta)_\mathbbm{b}^\mathbbm{c}(s) \cdot g_n(l_\mathbbm{a},l_\mathbbm{c})=(\mathcal {A}^{-1}_\delta)_\mathbbm{a}^\mathbbm{b}(s)$ and
\[
\underset{\mathbbm{a}}{\sum}g_T(\nabla_T^TY_\mathbbm{a},Y_\mathbbm{a})(s)=\text{tr}(\mathcal {A}'_\delta\cdot \mathcal {A}^{-1}_\delta)(s)=\frac{(\det \mathcal {A}_\delta)'}{\det \mathcal {A}_\delta}(s).\tag{*4}
\]
(*1) together with (*2), (*3) and (*4) furnishes
\[
\frac{(\det \mathcal {A})'}{\det \mathcal {A}}(s)\leq \frac{(\det \mathcal {A}_\delta)'}{\det \mathcal {A}_\delta}(s).
\]
Since $\det \mathcal {A}(t)\sim t^{n-k-1} \sim \det \mathcal {A}_\delta (t)$, $\det \mathcal {A}(t)\leq \det \mathcal {A}_\delta (t)$ for all $t\in [0,r]$, which implies that $c_f(\xi)\leq \zeta_0$. The arithmetic-geometric mean inequality now yields
\[
\det \mathcal {A}(t)\leq \det \mathcal {A}_\delta(t)\leq \left(\mathfrak{s}'_\delta-\frac{\underset{\alpha}{\sum}\lambda_\alpha}{k}\mathfrak{s}_\delta\right)^k(t)\cdot\mathfrak{s}^{n-k-1}_\delta(t),
\]
for $t\in [0, c_f(\xi)]$.\end{proof}

\section{proof of Theorem 1.1}
Given $\xi\in \mathcal {V}^*SN\backslash N$, choose a local coordinate system $(u^\alpha,\theta_\mathfrak{g})$ around $\xi$ such that $(t,u^\alpha,\theta_\mathfrak{g})$ is a local
cone coordinate system $(t,u^\alpha,\theta_\mathfrak{g})$ on $\mathcal {V}^*N\backslash N$. Set $\pi_1(\xi)=x$. It is easy to check that $\mathcal {L}^{-1}$ is an isometry from $(T^*_xM\backslash0,g^*_x)$ to $(T_xM\backslash0,g_x)$. Denote by $d\nu_x$ the Riemannian volume form on $\mathcal {V}^*_xSN$ induced by $g^*_x$. Let $n$ and $e_\mathbbm{a}$, $\mathbbm{a}=\alpha,\mathfrak{g}$ be defined as before. Since
\[
g^*_{\xi}\left(\frac{\partial}{\partial \theta_\mathfrak{g}},\frac{\partial}{\partial \theta_\mathfrak{h}}\right)=((\mathcal {L}^{-1})^*g_n)\left(\frac{\partial}{\partial \theta_\mathfrak{g}},\frac{\partial}{\partial \theta_\mathfrak{h}}\right)=g_n(e_\mathfrak{g},e_\mathfrak{h}),
\]
we have $d\nu_x(\xi)=\sqrt{\det g_{n}(e_\mathfrak{g},e_\mathfrak{h})}d\Theta$, where $d\Theta=\wedge_{\mathfrak{g}}d\theta_{\mathfrak{g}}$.
We define a $n$-form $\varpi$ on $(0,+\infty)\times\mathcal {V}^*SN$ by
\[
\varpi(t,\xi)=e^{-\tau(\dot{\gamma}_n(t))}\det \mathcal {A}(t,n)dt\wedge \sqrt{\det g_n(e_\alpha,e_\beta)}du^1\wedge\cdots\wedge du^k\wedge d\nu_x(\xi).
\]
A direct calculation shows that $\varpi$ is independent of the choice of chart.

Given an arbitrary point $p\in M\backslash N$, there exists a unit speed minimizing geodesic $\gamma_n$ from $N$ to $p$. By Remark 1, we have $\E(D)=M$, where $D:=\{(t,\xi): \xi\in \mathcal {V}^*SN ,\ 0\leq t\leq c_f(\xi)\}$.

Moreover, for each $x_0=E(t_0,\xi_0)\in M$ with $0<t<c_f(\xi_0)$, by Proposition 4.7, there exists an open set $\mathcal {Q}(t_0,\xi_0)=(t_0-\varepsilon,t_0+\varepsilon)\times \mathcal {W}(\xi_0)$ such that $\E|_{\mathcal {Q}(t_0,\xi_0)}:\mathcal {Q}(t_0,\xi_0)\rightarrow \E(\mathcal {Q}(t_0,\xi_0))$ is a diffeomorphism.
Hence, $(t\circ \E|_{\mathcal {Q}}^{-1},u^\alpha\circ \E|_{\mathcal {Q}}^{-1},{\theta}_\mathfrak{g}\circ \E|_{\mathcal {Q}}^{-1})$ is a local coordinate system on $\E(\mathcal {Q}(t_0,\xi_0))$. For simplicity, we still use $(t,u^\alpha,{\theta}_\mathfrak{g})$ or $(t,\xi)$ to denote this coordinate system. In this case,
\[
\left.\frac{\partial}{\partial t}\right|_{(t,\xi)}=P_{t;n}n,\ \left.\frac{\partial}{\partial u^\alpha}\right|_{(t,\xi)}=P_{t;n}\mathcal {A}e_\alpha,\ \left.\frac{\partial}{\partial {\theta}_\mathfrak{g}}\right|_{(t,\xi)}=P_{t;n}\mathcal {A} e_\mathfrak{g},
\]
where $n$ and $e_\mathbbm{a}$, $\mathbbm{a}=\alpha,\mathfrak{g}$ are defined as before.
And denote by $\det g_{\frac{\partial}{\partial t}}(\E(t,\xi))$ the determinant of $g_{\frac{\partial}{\partial t}}$.

Given a volume form $d\mu$ on $M$. It follows from (6) that
\begin{align*}
\E|_{\mathcal {Q}(t_0,\xi_0)}^*d\mu=:&\sigma(t,u^\alpha, \theta_\mathfrak{g})dt\wedge du^1\wedge\cdots\wedge du^k\wedge d\Theta\\
=&\frac{\sigma(t,u^\alpha, \theta_\mathfrak{g})}{\sqrt{\det g_{\frac{\partial}{\partial t}}(\E(t,\xi))}}\sqrt{\det g_{\frac{\partial}{\partial t}}(\E(t,\xi))}dt\wedge du^1\wedge\cdots\wedge du^k\wedge d\Theta\\
=&\varpi.\tag{4.2}
\end{align*}

\begin{proof}[ \textbf{Proof of Theorem 1.1}] By the argument above, we have $\mu(M)=\mu(D_d^0)$, where $D_d^0:=\{(t,\xi):\xi\in \mathcal {V}^*SN,\, 0< t\leq \min\{d,c_f(\xi)\} \}$.
Let $D_1=\{(t,\xi):\xi\in \mathcal {V}^*SN,\,0<t<d<c_f(\xi)\}$, $D_2:=\{(t,\xi):\xi\in \mathcal {V}^*SN,\,0<t<c_f(\xi)<d\}$ and $D_3:=\{(d,\xi): \xi\in \mathcal {V}^*SN,\,d=c_f(\xi)\}$. Sard's theorem then yields $\mu(M)=\mu(\E (D_1))+\mu(\E (D_2))$.

For each $(t,\xi)\in D_s$, $s=1,2$, from above, one can find an open neighborhood $\mathcal {Q}(t,\xi)\subset D_s$ of $(t,\xi)$ such that $\E|_{\mathcal {Q}(t,\xi)}:\mathcal {Q}(t,\xi)\rightarrow \E(\mathcal {Q}(t,\xi))$ is a diffeomorphism. Hence, there is a countable covering $\{\mathcal {Q}(t_i,\xi_i)\}$ of $D_1\cup D_2$. For simplicity, set $\mathcal {Q}_i:=\mathcal {Q}(t_i,\xi_i)$ and $E_i:=E|_{\mathcal {Q}_i}$. Note that $\{E(\mathcal {Q}_i)\}$ is also a open covering of $E(D_1\cup D_2)$. Let $\{\rho_i\}$ be a partition of unity subordinate to $\{E(\mathcal {Q}_i)\}$. And define a sequence of nonnegative continuous functions $\varrho_i:D_1\cup D_2\rightarrow \mathbb{R}$ by
\[
\varrho_i(t,\xi):=\left \{
\begin{array}{ll}
&\rho_i\circ \E_i, \  (t,\xi)\in \mathcal {Q}_i\\
&0,\ \text{else},
\end{array}
\right.
\]
Given $(t,\xi)\in D_1\cup D_2$,
\[
\underset{i}{\sum}\varrho_i(t,\xi)=\underset{\{i:(t,\xi)\in \mathcal {Q}_i\}}{\sum}\varrho_i(t,\xi)=\underset{\{i:(t,\xi)\in \mathcal {Q}_i\}}{\sum}\rho_i(E(t,\xi))\leq \underset{i}{\sum}\rho_i(E(t,\xi))=1,
\]
From above, we have
\begin{align*}
\mu(M)&=\underset{i}{\sum}\int_{\E_i(\mathcal {Q}_i)}\rho_i\cdot d\mu=\underset{i}{\sum}\int_{\mathcal {Q}_i}(\rho_i\circ\E_i)\cdot\varpi\\
&=\underset{\{i:\mathcal {Q}_i\subset D_1\}}{\sum}\int_{\mathcal {Q}_i}\varrho_i\cdot\varpi+\underset{\{i:\mathcal {Q}_i\subset D_2\}}{\sum}\int_{\mathcal {Q}_i}\varrho_i\cdot\varpi\\
&\leq \underset{\{i:\mathcal {Q}_i\subset D_1\}}{\sum}\int_{D_1}\varrho_i\cdot\varpi+\underset{\{i:\mathcal {Q}_i\subset D_2\}}{\sum}\int_{D_2}\varrho_i\cdot\varpi\\
&\leq \int_{D_1}\varpi+\int_{D_2}\varpi= \int_{D_1\cup D_2}\varpi.\tag{7}
\end{align*}

\textbf{Case 1:} $k=0$, i.e., $N=\{x\}$. Note that $\mathcal {V}^*_xSN=S^*_xM$ and $\mathcal {L}:(T_xM\backslash0,g_x)\rightarrow (T^*_xM\backslash0,g^*_x)$ is an isometry. For convenience, we also use $d\nu_x$ to denote the Riemannian volume form induced by $g_x$ on $S_xM$. It now follows Theorem 4.8 that
\[
\mu(M)\leq \int_{S_xM}e^{-\tau(\dot{\gamma}_y(t))}d\nu_x(y) \int^{d}_0\mathfrak{s}^{m-1}_\delta(t)dt.
\]

\textbf{Case 2:} Let $d\mu$ and $d\bar{\mu}$ denote either the Busemann-Hausdorff volume forms or Holmes-Thompson volume forms induced by $F$ and $F|_N$, respectively. Set $d\bar{\mu}=:\sigma_Ndu^1\wedge\cdots \wedge du^k$.
Hence,
\begin{align*}
\mu(M)
&\leq \int^{\min\{d,c_f(\xi)\}}_0 \left(\mathfrak{s}'_\delta-\frac{H_\xi}{k}\mathfrak{s}_\delta\right)^k(t)\cdot\mathfrak{s}^{m-k-1}_\delta(t)dt\int_{\mathcal {V}^*_xSN}e^{-\tau\left(\dot{\gamma}_{\mathcal {L}^{-1}(\xi)}(t)\right)} d\nu_x(\xi)\\
&\cdot\int_N \sqrt{ \det  g_{\mathcal {L}^{-1}(\xi)}\left(\frac{\partial}{\partial u^\alpha},\frac{\partial}{\partial u^\beta}\right)}du^1\cdots du^k(x).
\end{align*}
Choose $\xi_0\in \mathcal {V}^*SN$ such that $H_{\xi_0}=\min_{\xi\in \mathcal {V}^*SN}H_\xi$. Denote by $\zeta(\xi)$ the first positive zero of $\left(\mathfrak{s}'_\delta-\frac{H_{\xi}}{k}\mathfrak{s}_\delta\right)(t)$. Theorem 4.8 yields that $c_f(\xi)\leq \zeta(\xi)\leq \zeta(\xi_0)$ for all $\xi\in \mathcal {V}^*SN$.
By Proposition 7.1 and Proposition 7.3, we have
\begin{align*}
\mu(M)\leq c_{m-k-1}\cdot\Lambda_F^{(3m+k)/2}\cdot\bar{\mu}(N)\cdot\int^{\min\left\{d,\,\zeta(\xi_0)\right\}}_0
\left(\mathfrak{s}'_\delta-\frac{H_{\xi_0}}{k}\mathfrak{s}_\delta\right)^k(t)\cdot\mathfrak{s}^{m-k-1}_\delta(t)dt.
\end{align*}
\end{proof}

\begin{proof}[\textbf{Proof of Corollary 1.2}]
Without loss of generality, we can suppose that $N$ is a unit speed geodesic. Then for each $\xi\in \mathcal {V}^*SN$, we have
\[
-H_\xi=[g_n(X,X)]^{-1}\mathbf{T}_n(X)\leq l\frac{\left[\sqrt{g_n(X,X)}-g_n(X,n)\right]^2F(n)}{g_n(X,X)}= l,
\]
where $n=\mathcal {L}^{-1}(\xi)$ and $X=\frac{\partial}{\partial u}$ is the tangent vector field of $N$. Then we have
\[
\mu(M)\leq c_{m-2}\cdot\Lambda_F^{(3m+1)/2}\cdot\ell\cdot\left[\frac{\mathfrak{s}^{m-1}_\delta\left({\min\left\{d,\frac{\pi}{2\sqrt{\delta}}\right\}}\right)}{m-1}+ l\cdot\int_0^{d}\mathfrak{s}^{m-1}_\delta(t) dt\right].
\]

\end{proof}

\section{Randers metric}
Let $(M,F)$ be a compact Randers manifold and let $\gamma(t)$, $0\leq t\leq \ell$, be a unit speed closed geodesic in $M$. Set $b:=\sup_{x\in \gamma([0,\ell])}\|\beta\|_{\alpha}$ and $b_1:=\sup_{x\in \gamma([0,\ell])}\|\nabla\beta\|_{\alpha}$.

For each $t$, there exits a local coordinate $(u,x^A)$ such that $\frac{\partial}{\partial u}=\dot{\gamma}(t)$ and $x^A|_{\gamma(t)}=0$. Given $\xi\in \mathcal {V}_{\gamma(t)}SN$, let $n=\mathcal {L}^{-1}(\xi)$. It is easy to check that $e^{-\tau_{BH}(\dot{\gamma}_n(t))}\leq (1+b)^{(m+1)/2}$ and $e^{-\tau_{HT}(\dot{\gamma}_n(t))}\leq (1-b)^{-(m+1)/2}$.

From Example 1, we have
\[
g_n\left(\frac{\partial}{\partial u},\frac{\partial}{\partial u}\right)=\frac{1}{\alpha(n)}\left[\alpha^2\left(\frac{\partial}{\partial u}\right)-\frac{\langle\frac{\partial}{\partial u},n\rangle^2}{\alpha^2(n)}\right]=\frac{F(-\frac{\partial}{\partial u})}{1-\beta(n)}\leq \frac{1+b}{(1-b)^2}.
\]
By the formula of T-curvature, we have
\[
T_n\left(\frac{\partial}{\partial u}\right)\leq \frac{b_1(2b^3+5b^2-2b+7)}{2(1-b)^3}
\]
Hence, we have the following
\begin{theorem}
Let $(M,F)$ be a compact Randers manifold with $\mathbf{K}\geq \delta$ and let $\gamma$ be a closed geodesic in $M$. Set $b:=\sup_{x\in M}\|\beta\|_{\alpha}$ and $b_1:=\sup_{x\in M}\|\nabla\beta\|_{\alpha}$. Then
\[
L_F(\gamma)\geq \frac{(1-b)^{\frac{m+2}{2}}}{c_{m-2}(1+b)^{\frac{1}{2}}\mathfrak{S}(b,b_1,\delta,d,m)}\max\left\{\frac{\mu_{BH}(M)}{(1+b)^{\frac{m+1}{2}}},(1-b)^{\frac{m+1}{2}}\vol_\alpha(M)\right\},
\]
where
\[
\mathfrak{S}(b,b_1,\delta,d,m)=\frac{\mathfrak{s}_\delta^{m-1}\left(\min\left\{d,\frac{\pi}{2\sqrt{\delta}}\right\}\right)}{m-1}+\frac{b_1(2b^3+5b^2-2b+7)}{2(1-b)^3}\int_0^{d}\mathfrak{s}^{m-1}_\delta(t) dt,
\]
and $\vol_\alpha$ is the Riemannian volume of $M$ induced by $\alpha$.
\end{theorem}
\begin{remark}
In fact, we can take $b:=\sup_{x\in \gamma}\|\beta\|_{\alpha}$ and $b_1:=\sup_{x\in \gamma}\|\nabla\beta\|_{\alpha}$.
\end{remark}

\section{Non-Riemannian example}
In \cite{CH}, Cheeger showed the existence of the lower bound for the length of sample close geodesics in a closed Riemannian manifold in terms of an upper bound for the diameter and lower bounds for the volume and the curvature. However, this is not true in Finsler geometry. First, we introduce the notations in this section.

Let $M=\mathbb{S}^n\times\mathbb{S}$ $(n\geq2)$, and $\alpha$ be the canonical Riemannian product metric on $M$. Let $\beta_\epsilon$ be as in Example 1. Choose a smooth function $\phi$ defined on $(-1,1)$ such that
\begin{align*}
\text{(i)}&\underset{s\rightarrow -1}{\lim}\phi(s)=0,\ \underset{s\in (-1,1)}{\sup}\phi(s)<+\infty;\\
\text{(ii)}&\phi(s)-s\phi'(s)+(t^2-s^2)\phi''(s)>0,\ \text{for }t\in [0,1), \|s\|\leq t<1.
\end{align*}
Define $F_\epsilon:=\alpha\phi(s)$, $s=\beta/\alpha$.
It follows from \cite{CS} that $F_\epsilon$ is a Finsler metric on $M$ for all $\epsilon\in [0,1)$.
Note that $\beta$ is parallel corresponding $\alpha$. A direct calculation shows the spray of $F_\epsilon$ is the same as the one of $\alpha$. Hence, it is easy to check that the flag curvature $\mathbf{K}_{F_\epsilon}\geq 0$ for all $\epsilon\in [0,1)$. Let $(r,\theta^\alpha,t)$ be the local coordinate system of $M$, where $(r,\theta^\alpha)$ (resp. $t$) is the spherical coordinates of $\mathbb{S}^n$ (resp. $\mathbb{S}$). Thus, $\gamma(t)=(0,0,-t)$ is a geodesic of $F_\epsilon$. (i) yields that $L_{F_\epsilon}(\gamma)\rightarrow 0$, where $\epsilon\rightarrow 1$.
For convergence, set
\[
\mathscr{T}(s):=\phi(s)\cdot(\phi(s)-s\phi'(s))^{n-1}[\phi(s)-s\phi'(s)+(\epsilon^2-s^2)\phi''(s)].
\]

Then we have the following
\begin{theorem}Let $(M,F_\epsilon)$ be as above. Suppose one of the following conditions is true:
 \begin{align*}
\text{(1)}&\underset{\epsilon\rightarrow 1}{\lim}\phi(\epsilon\cos t)\geq C_1, \text{for almost every }t\in [0,\pi];\\
\text{(2)}&\underset{\epsilon\rightarrow 1}{\lim}\mathscr{T}(\epsilon\cos t)\geq C_2, \text{for almost every }t\in [0,\pi];\\
\text{(3)}&\varphi(s):=\mathscr{T}(s)-1\text{ is odd function};
\end{align*}
Here $C_1$ and $C_2$ are positive constants. Then $\mathbf{K}_\epsilon\geq 0$, $\mu_\epsilon(M)\geq V$ and $\diam_\epsilon(M)\leq D$ for all $\epsilon\in [0,1)$, where $\mu_\epsilon$ denote the

there exists a geodesic $\gamma$ of $(M,F_\epsilon)$ such that

\end{theorem}

(i) yields that

Let $M$, $\alpha$ and $\beta_\epsilon$ be as in Example 1. Let $\phi$ be a smooth positive function on $[-1,1]$ such that $\phi(s)-s\phi'(s)+(t^2-s^2)\phi{''}(s)>0$, for all $t\in [0,1)$ and $|s|\leq t<1$. Define a function $F_\epsilon:TM\rightarrow [0,+\infty)$ by
\[
F_\epsilon:=\alpha\phi(s), \ s=\frac{\beta_\epsilon}{\alpha}.
\]
By \cite{CS}, $F_\epsilon$ is a Berwald metric for all $\epsilon\in [0,1)$.
Since $M$ is compact and $\phi$ is defined on $[-1,1]$, $\diam_{F_\epsilon}(M)\leq \diam_\alpha(M)\cdot\max_{s\in[-1,1]}\phi(s)$. By \cite{Ce,CS}, it is easy to check that $\mathbf{K}_{F_\epsilon}\geq 0$ and $\mu_{\epsilon}(M)=C$, where $C$ is a positive constant only dependent on $\phi$, and $\mu_\epsilon$ denote either the Busemann-Hausdorff volume or the Holmes-Thompson volume of $(M,F_\epsilon)$.

\section{Appendix}
\begin{proposition}
Let $(M,F)$ be a Finsler $m$-manifold with finite uniform constant $\Lambda_F$. Let $d\mu$ denote either the Busemann-Hausdorff volume form or the Holmes-Thompson volume form on $M$. Then the distortion $\tau$ of $d\mu$ satisfy
$e^{-\tau(y)}\leq \Lambda_F^m$,
for all $y\in SM$.
\end{proposition}
\begin{proof}
Given $z\in T_xM\backslash 0$, let $(x^i)$ be local coordinates around $x$ and let $B_F(x):=\{y=y^i\frac{\partial}{\partial x^i}:F(x,y)<1\}$.
Set $\det g_{ij}(x,z_1):=\max_{y\in S_xM}\det g_{ij}(x,y)$ and $\det g_{ij}(x,z_2):=\min_{y\in S_xM}\det g_{ij}(x,y)$. By \cite[Proposition 3.1,Proposition 4.1]{W2}, we have
\[
\frac{\det g_{ij}(x,z_1)}{\det g_{ij}(x,z_2)}\leq \Lambda^m_F,\ \Lambda^{-m/2}_F\leq \frac{\vol_{g_{z_1}} (B_F(x))}{\vol(\mathbb{B}^m)}\leq \Lambda^{m/2}_F.
\]

(1): Suppose $d\mu$ is the Busemann-Hausdorff volume form.
Hence,
\begin{align*}
\sqrt{\det g_{ij}(x,z)}\vol(B_F(x))&\geq \sqrt{\frac{\det g_{ij}(x,z_2)}{\det g_{ij}(x,z_1)}}\int_{B_F(x)}\sqrt{\det g_{ij}(x,z_1)}dy^1\wedge\cdots\wedge dy^m\\
&\geq \Lambda^{-m}_F\vol(\mathbb{B}^m),
\end{align*}
where $B_F(x):=\{y\in T_xM:F(x,y)<1\}$.
 Then
\[
e^{-\tau(z)}=\frac{\vol(\mathbb{B}^m)}{\sqrt{\det g_{ij}(x,z)}\vol (B_F(x))}\leq\Lambda^m_F.
\]

(2): Suppose $d\mu$ is the Holmes-Thompson volume form. Thus,
\begin{align*}
e^{-\tau(z)}&= \frac{1}{\vol(\mathbb{B}^m)}\frac{\int_{B_F(x)}\det g_{ij}(x,y)dy^1\wedge\cdots\wedge dy^n}{\sqrt{\det g_{ij}(x,z)}}\\
&\leq \sqrt{\frac{\det g_{ij}(x,z_1)}{\det g_{ij}(x,z_2)}}\frac{\vol_{g_{z_1}}(B_F(x))}{\vol(\mathbb{B}^m)}\leq \Lambda_F^m.
\end{align*}
\end{proof}

\begin{proposition}
Let $(M,F)$ be a Finsler manifold with finite uniform constant $\Lambda_F$. Then the uniform constant of $F^*$ is still $\Lambda_F$.
\end{proposition}
\begin{proof}Let $\Lambda_{F^*}$ denote the the uniform constant of $F^*$.
Given $y,z\in S_xM$. For each $X\in T_xM\backslash0$, we have
\[
\Lambda_F^{-1}\leq \frac{g_y(X,X)}{g_z(X,X)}\leq \Lambda_F.
\]
Choose a $g_z$-orthonormal basis $\{e_i\}$ for $T_xM$ consisting of eigenvectors
of $g_y$ with respective eigenvalues $\lambda_i$. Then we have
\[
\Lambda_F^{-1}\leq\frac{\sum_i\lambda_i (X^i)^2}{\sum_i (X^i)^2}\leq \Lambda_F,
\]
where $X=X^ie_i$. Hence, we have $\Lambda_F^{-1}\leq \lambda_i\leq \Lambda_F$, for all $i$. Therefore,
\[
\Lambda_F^{-1}\leq\frac{\sum_i\lambda^{-1}_i (\xi_i)^2}{\sum_i (\xi_i)^2}\leq \Lambda_F,
\]
where $\xi=(\xi_i)\neq0$. Let $\omega^i$ be the dual basis of $\{e_i\}$. Note that
\[
g^*_{\mathcal {L}(y)}=g^{ij}_{y}e_i\otimes e_j=\sum_i\lambda^{-1}_i e_i\otimes e_i,\ g^*_{\mathcal {L}(z)}=\sum_i  e_i\otimes e_i.
\]
Hence, we have
\[
\Lambda_F^{-1}\leq \frac{g^*_{\mathcal {L}(y)}(\xi,\xi)}{g^*_{\mathcal {L}(z)}(\xi,\xi)}\leq \Lambda_F,
\]
for all $\xi\in T^*M\backslash 0$, which implies that $\Lambda_{F^*}\leq \Lambda_F$. Likewise, one can show $ \Lambda_F\leq \Lambda_{F^*}$.
\end{proof}

\begin{proposition}
Let $(M,F)$ be a Finsler $m$-manifold with finite uniform constant $\Lambda_F$ and $N$ be $k$-dimensional submanifold of $M$. For each $x\in N$, we have
$\nu_x(\mathcal {V}_x^*SN)\leq c_{m-k-1}\cdot\Lambda^{(m-k)/2}_F$. Moreover, if $F=\alpha+\beta$ is a Randers metric, then $\nu_x(\mathcal {V}_x^*SN)\leq{c_{m-k-1}}\cdot{(1-b(x))^{-\frac{m-k+1}{2}}}$, where $b(x):=\|\beta\|_\alpha(x)$.
\end{proposition}
\begin{proof}

For each $x\in N$, let $(u^\alpha,\xi_A)$ be a canonical coordinate system around $\mathcal {V}^*_xN$. Then we have
\[
\mathcal {V}_x^*SN=\{\xi=\xi_A dx^A: F^*(x,\xi)=1\}.
\]
Clearly,
\[
d\nu_x(\xi)=\sqrt{\det g^{*AB}_\xi}\left(\underset{A}{\sum}(-1)^{A+1}\xi_Ad\xi_{k+1}\wedge\cdots d\hat{\xi}_{A}\wedge\cdots d\xi_m\right).
\]
Set $\mathcal {V}^*_xBN:=\{\xi=\xi_A dx^A: F^*(x,\xi)<1\}$. Then we have
\[
\vol_{g^*_x}(\mathcal {V}^*_xBN)=\int_{\mathcal {V}^*_xBN}\sqrt{\det g^{*AB}_\xi}d\xi_{k+1}\wedge\cdots \wedge d\xi_{m}=\frac{1}{m-k}\int_{\mathcal {V}^*_xSN}d\nu_x.
\]
A sample argument based on \cite[Proposition4.1]{W2} and Proposition 6.2 shows $\nu_x(\mathcal {V}_x^*SN)\leq c_{m-k-1}\cdot\Lambda^{(m-k)/2}_F$.

If $F=\alpha+\beta$ is a Randers metric, then $F^*=\alpha^*+\beta^*$ is also a Randers metric. Denote by $\Sigma$ the subspace $\{(\xi_\alpha, \xi_A): \xi_\alpha=0,\ \forall \alpha\}$ of $T^*_xM$. And set $F^*|_{\Sigma}=:\bar{F}^*=:\bar{\alpha}^*+\bar{\beta}^*$. Clearly, $\frac12 (\bar{F}^*)_{\xi_A\xi_B}(\xi)=g^{*AB}_\xi$, where $\xi\in \Sigma\backslash 0$. By \cite{Sh1}, we have
\[
\underset{\xi\in \Sigma\backslash0}{\sup}\frac{\bar{\beta}^*(\xi)}{\bar{\alpha}^*(\xi)}\leq \underset{\xi\in T^*_xM\backslash0}{\sup}\frac{\beta^*(\xi)}{\alpha^*(\xi)}=\|\beta^*\|_{\alpha^*}=b(x),
\]
which implies that
\[
\det g^{*AB}_\xi=(\det\alpha^{*AB})\left(\frac{\bar{F}^*(\xi)}{\bar{\alpha}^*(\xi)}\right)^{m-k+1} \leq (\det \alpha^{*AB})(1+b(x))^{m-k+1}.
\]

By using the argument given in \cite{Sh1}, we obtain
\[
\int_{\mathcal {V}^*_xSN}d\nu_x\leq\frac{c_{m-k-1}}{(1-b(x))^{\frac{m-k+1}{2}}} .
\]
\end{proof}

\begin{proposition}
Let $F=\alpha+\beta$ be a Randers metric, were $\alpha(y)=\sqrt{a_{ij}y^iy^j}$ and $\beta(y)=b_iy^i$. Let $b_{i|j}$ denote the covariant derivative corresponding with $\alpha$. Set
\[
r_{ij}=\frac12 (b_{i|j}+b_{j|i}),\ s_{ij}:=\frac12 (b_{i|j}-b_{j|i}),\ s^i_{\,j}:=a^{ik}s_{kj},\ s_j:=b_is^i_{\,j},\ e_{ij}:=r_{ij}+b_is_j+b_js_i.
\]
The we have the following
\begin{align*}
T_y(v)
=&\left[-2\left(\frac{e_{11}}{2F(v)}-s_1\right)+2\frac{s_{01}}{\alpha(y)}+\frac{1}{\alpha(y)}\left(\frac{e_{00}}{2F(y)}-s_0\right)\left(\alpha(v)+\frac{\langle v,y\rangle}{\alpha(y)}\right)\right]F(y)\\
&\cdot\left(\frac{\alpha(v)\alpha(y)-\langle v,y\rangle}{\alpha(y)}\right),
\end{align*}
where the index "0" (resp. "1") means the contraction with $y^i$ (resp. $v^i$).
\end{proposition}

\end{document}